\documentclass[12pt]{article}
\usepackage{booktabs}
\usepackage{caption}
\usepackage{mathrsfs}
\usepackage{amsmath}
\usepackage{amsfonts,amsthm,amssymb,mathrsfs,bbding}
\usepackage{graphics,multicol}
\usepackage{graphicx}
\usepackage{color}
\usepackage{enumerate}
\usepackage{caption}
\usepackage{rotating}
\usepackage{lscape}
\usepackage{longtable}

\makeatletter

\newcommand{\Rmnum}[1]{\expandafter\@slowromancap\romannumeral #1@}
\makeatother

\usepackage{amssymb}
\allowdisplaybreaks[4]
\usepackage{tabularx}                                                                                                               
\usepackage[colorlinks=true,anchorcolor=blue,filecolor=blue,linkcolor=blue,urlcolor=blue,citecolor=blue]{hyperref}
\usepackage{extarrows}
\usepackage{cite}
\usepackage{latexsym,bm}
\usepackage{mathtools}
\evensidemargin=\oddsidemargin\topmargin=-1.5cm

\newtheorem{thm}{Theorem}

\newtheorem{lem}{Lemma}
\newtheorem{cor}{Corollary}

\newtheorem{claim}{Claim}
\theoremstyle{definition}

\addtocounter{section}{0}
\begin{document}

	\title{\bf A spectral radius condition for a graph to have  $(a,b)$-parity factors}
	\author{{Junjie Wang$^{a}$, Yang Yu$^{a,}$\footnote{Corresponding author.}\setcounter{footnote}{-1}\footnote{\emph{E-mail address:} understand8106@163.com}, Jianbiao Hu$^{a}$}, Peng Wen$^{a}$\\[2mm]
		\small $^{a}$School of Mathematics, East China University of Science and Technology,\\
		\small Shanghai 200237, China }

	\date{}
	\maketitle
	{\flushleft\large\bf Abstract }  Let $a,b$ be two positive integers such that $a \le b$ and $a \equiv b$ (mod $2$). We say that a graph $G$ has an $(a,b)$-parity factor if $G$ has a spanning subgraph $F$ such that $d_{F}(v) \equiv b$ (mod $2$) and $a \le d_{F}(v) \le  b$ for all $v \in V (G)$. In this paper, we provide a tight spectral radius condition for a graph to have  $(a,b)$-parity factors. 
	
	\begin{flushleft}
	\end{flushleft}

	\section{Introduction}
	All graphs considered in this paper are simple and undirected. Let $G$ be a graph with vertex set $V(G)$ and  edge set $E(G)$, and let $e(G)=|E(G)|$ denote the size of $G$.  For any $v\in V(G)$, let $d_G(v)$ denote the degree of $v$ in $G$, and let $\delta(G)=\min_{v\in V(G)}d_G(v)$. For any $S \subseteq V(G)$, we denote by $G[S]$ the subgraph of $G$ induced by $S$, and write $E_G(S)=E(G[S])$ and $e_G(S)=|E_G(S)|$. Also, we denote by $E_G(S,T)$ the set of edges between $S$ and $T$ for any two disjoint subsets $S$ and $T$ of $V(G)$, and write $e_G(S,T)=|E_G(S,T)|$. The \textit{adjacency matrix} of $G$ is defined as $A(G)=(a_{u,v})_{u,v\in V(G)}$, where $a_{u,v}=1$ if $u$ and $v$ are adjacent in $G$, and $a_{u,v}=0$ otherwise. The eigenvalues of $A(G)$ are called the \textit{eigenvalues} of $G$, and denoted by $\lambda_1(G)\geq \cdots\geq \lambda_n(G)$, where $n=|V(G)|$. In particular,  the \textit{spectral radius} of $G$ is defined by $\rho (G):=\lambda_1(G)$. For some basic results on the spectral radius of graphs, we refer the reader to \cite{Ho, St, WXH}, and references therein.

In the past half century, the theory of graph factors played a key role in the study of graph theory \cite{AK,Co,KT,Lo1,LZ,T,TT}. Let $g$ and $f$ be two integer-valued functions defined on $V(G)$ such that $0 \le g(v) \le f(v)$ for all $v \in V (G)$. A \textit{$(g,f)$-factor} of $G$ is a spanning subgraph $F$ of $G$ satisfying $g(v) \le d_{F}(v) \le f(v)$ for all $v\in V(G)$. In particular, an $(f,f)$-factor is called an \textit{$f$-factor}. 
	Let $a\leq b$ be two positive integers. If $g\equiv a$ and $f\equiv b$, then a $(g,f)$-factor is also called an $[a,b]$-\textit{factor}.
	
In 1952, Tutte \cite{T} gave a necessary and sufficient condition for the existence of an $f$-factor in a graph, which is known as Tutte's $f$-factor theorem. In 1970, Lov\'{a}sz \cite{Lo1} generalized  the conclusion of Tutte's $f$-factor theorem to $(g,f)$-factors.

	\begin{thm}(Lov\'{a}sz\cite{Lo1})\label{thm::1}
		A graph $G$ has a $(g, f)$-factor if and only if for any two disjoint subsets $S$, $T$ of $V(G)$,
		$$
		f(S)-g(T)+\sum_{x\in T}d_{G-S}(x)-\hat{q}_{G}(S,T) \ge 0,
		$$
		where $\hat{q}_{G}(S,T)$ denotes the number of components $C$ in $G-S-T$ such that $g(v)=f(v)$ for all $v \in V(C)$ and $f(V(C))+e_{G}(V(C), T) \equiv 1 \pmod 2$.
	\end{thm}

	Based on Theorem \ref{thm::1}, a series of sufficient conditions in terms of vertex degrees or eigenvalues have been obtained for the existence of a $(g,f)$-factor (or  particularly, $[a,b]$-factor) in  graphs \cite{A, CHO, FLL, OS2, WZ, Xi}. 
	
	Suppose that $g(v) \equiv f(v) \pmod{2}$ for all $v \in V(G)$. We say that $G$ has a \textit{$(g,f)$-parity factor} if it has a spanning subgraph $F$ such that $d_{F} (v) \equiv f(v)$ (mod $2$) and $g(v) \le d_{F} (v) \le f(v)$ for all $v \in V (G)$. In particular, if $g(v)=a$ and $f(v)=b$ for all $v \in V(G)$, where $a,b$ are positive integers such that $a \le b$ and $a \equiv b~(\mod 2)$, then a $(g, f)$-parity factor is called a $(a,b)$-\textit{parity factor}. In \cite{Lo}, Lov\'{a}sz provided a characterization for graphs having $(g, f)$-parity factors.

	\begin{thm}(Lov\'{a}sz \cite{Lo})\label{thm::3}
		 A graph $G$ has a $(g, f)$-parity factor if and only if for any two disjoint subsets $S$, $T$ of $V(G)$,
		\begin{equation}
			\eta (S,T) = f(S)-g(T)+\sum_{x \in T} d_{G-S}(x)-q(S,T) \ge 0,
			\nonumber
		\end{equation}
		where $q(S,T)$ denotes the number of components $C$ in $G-S-T$ such that $g(V(C))+e_{G}(V(C),T) \equiv 1 \pmod 2$.
	\end{thm}

    As an immediate corollary, we have the following result.
    \begin{cor}\label{cor::1}
    	Let $a,b,n$ be three positive integers such that $a\leq b$, $a\equiv b\pmod 2$, and $na$ is even. Let $G$ be a graph of order $n$. Then $G$ has an $(a,b)$-parity factor if and only if for any two disjoint subsets $S$, $T$ of $V(G)$,
    	\begin{equation}\label{equ::1}
    		\eta (S,T) = b|S|-a|T|+\sum_{x \in T} d_{G-S}(x)-q(S,T) \ge 0,
    	\end{equation}
    	where $q(S,T)$ denotes the number of components $C$ in $G-S-T$ such that $a|V(C)|+e_{G}(V(C),T) \equiv 1 \pmod 2$.
    \end{cor}

 Using Corollary \ref{cor::1}, Liu and Lu \cite{LL} obtained a degree condition for a graph to have $(a,b)$-parity factors.

    \begin{thm}(Liu and Lu\cite{LL})
        Let $a$, $b$, $n$ be three positive integers such that $a \equiv b \pmod{2}$, $na$ is even and $n \ge b(a+b)(a+b+2)/(2a)$. Let $G$ be a connected graph of order $n$. If $\delta (G) \ge a+(b-a)/a$ and 
        $$
        \max \{ d_G(u), d_G(v) \} \ge \frac{an}{a+b}
        $$
        for any two non-adjacent vertices $u$ and $v$ in $G$, then $G$ has an $(a,b)$-parity factor.
    \end{thm}

In recent years, some  sufficient conditions in terms of the eigenvalues are also established for the existence of an $(a,b)$-parity factor in graphs. For positive integers $r\geq 3$ and even $n$, Lu, Wu, and Yang \cite{LWY}  proved a lower bound for $\lambda_3(G)$ in an $n$-vertex $r$-regular graph $G$ to guarantee the existence of
a $(1, b)$-parity factor in $G$, where $b$ is a positive odd integer.  In \cite{KO}, Kim, O, Park and Ree improved this bound, and proved the following result.

    \begin{thm}(Kim, O, Park and Ree\cite{KO})
        Let $k \ge 3$, and $b$ be a positive odd integer less than $k$. If $\lambda_{3} (G)$ of a $k$-regular graph $G$ with even number of vertices is smaller than $\rho(k,b)$, where 
        \begin{equation}
		\rho(k,b)=\left\{
		\begin{array}{ll}
			\frac{k-2+\sqrt{(k+2)^2-4(\lceil \frac{k}{b} \rceil-2)}}{2},&\mbox{if both $r$ and $\lceil \frac{k}{b} \rceil$ are even},\\
            \frac{k-2+\sqrt{(k+2)^2-4(\lceil \frac{k}{b} \rceil-1)}}{2},&\mbox{if $r$ is even and $\lceil \frac{k}{b} \rceil$ is odd},\\
			\frac{k-3+\sqrt{(k+3)^2-4(\lceil \frac{k}{b} \rceil-2)}}{2},&\mbox{if both $r$ and $\lceil \frac{k}{b} \rceil$ are odd},\\
            \frac{k-3+\sqrt{(k+3)^2-4(\lceil \frac{k}{b} \rceil-1)}}{2},&\mbox{if $r$ is odd and $\lceil \frac{k}{b} \rceil$ is even,}
		\end{array}
		\right.
        \nonumber
		\end{equation}
        then $G$ has a $(1, b)$-parity factor.
    \end{thm}

    Motivated by the above works, in this paper, we provide a tight spectral radius condition for a graph to have an $(a,b)$-parity factor. The \textit{join} of two graphs $G_1$ and $G_2$, denoted by $G_1\nabla G_2$, is the graph obtained from the vertex-disjoint union $G_1\cup G_2$ by adding all possible edges between $G_1$ and $G_2$. In particular, we denote $H_{n,a}=K_{a-1}\nabla (K_1\cup K_{n-a})$, where $K_s$ is the complete graph of order $s$. Our main result is as follows.
\begin{thm}\label{thm::5}
    	Let $a,b,n$ be three positive integers such that $a\leq b$, $a\equiv b\pmod 2$, and $na$ is even. If $G$ is a graph of order $n$ with $n \ge \max \{2a(a+1)+b-3, (2a+2)(a+1) \}$ and $\rho (G)\geq \rho (H_{n,a})$, then $G$ has an $(a,b)$-parity factor unless $G\cong H_{n,a}$.
\end{thm}

	\section{Preliminaries}

In this section, we introduce some notions and lemmas, which are useful in the proof of Theorem \ref{thm::5}.

Let $M$ be a real $n$ $\times$ $n$ matrix, and let $\Pi=\{X_{1},X_{2}, \ldots,X_{k}\}$ be a partition of $[n]=\{1,2,\ldots,n\}$. Then the matrix $M$ can be written as
$$
M=\left(\begin{array}{ccccccc}
	M_{1,1}&M_{1,2}&\cdots&M_{1,k}\\
	M_{2,1}&M_{2,2}&\cdots&M_{2,k}\\
	\vdots&\vdots&\ddots&\vdots\\
	M_{k,1}&M_{k,2}&\cdots&M_{k,k}\\
\end{array}\right).
$$
The \textit{quotient matrix} of $M$ with respect to $\Pi$ is  the matrix $B_{\Pi}=(b_{i,j})^{k}_{i,j=1}$ with
$$
b_{i,j}=\frac{1}{|X_{i}|}\mathbf j^{T}_{|X_{i}|}M_{i,j}\mathbf j_{|X_{j}|}
$$
for all $i,j \in \{ 1,2,...,k \}$, where $\mathbf j_{s}$ denotes the all ones vector in $\mathbb{R}^{s}$. If each block $M_{i,j}$ of $M$ has constant row sum $b_{i,j}$, then $\Pi$ is called an \textit{equitable partition}, and the quotient matrix $B_\Pi$ is called an  \textit{equitable quotient matrix} of $M$. Also, if the eigenvalues of $M$ are all real, we denote them by $\lambda_1(M)\geq \lambda_2(M)\geq \cdots \geq \lambda_n(M)$. 

\begin{lem}\label{lem::1}(Brouwer and Haemers \cite[p. 30]{BH}; Godsil and Royle \cite[pp.196--198]{GR})
	Let $M$ be a real symmetric matrix, and let $B$ be an equitable quotient matrix of $M$. Then the eigenvalues of $B$ are also eigenvalues of $M$. Furthermore, if $M$ is nonnegative and irreducible, then
	$$
	\lambda_1(M)=\lambda_1(B).
	$$
\end{lem}

\begin{lem}\label{lem::2}(Hong \cite{Ho})
	Let $G$ be a connected graph with $n$ vertices and $m$ edges. Then
	$$\rho(G)\le \sqrt{2m-n+1}.$$
\end{lem}

 \begin{lem}\label{lem::3}
Let $a,b,n$ be three positive integers such that $a\leq b$, $a\equiv b\pmod 2$, and $na$ is even.  Then  $H_{n,a}$ has no $(a,b)$-parity factors.
\end{lem}
\begin{proof}
	Recall that $H_{n,a}=K_{a-1}\nabla (K_1\cup K_{n-a})$. Let $V_1=V(K_1)$, $V_2=V(K_{a-1})$ and $V_3=V(K_{n-a})$.  Take $S=\emptyset$ and $T=V_1$. Clearly, $q(S,T)=1$. Then we have
	$$\eta_G(S,T)=b|S|-a|T|+\sum_{x\in T}d_{G-S}(x)-q(S,T)=-a+a-1-1=-2<0,$$
	contrary to the inequality in \eqref{equ::1}. Therefore, by Corollary \ref{cor::1}, we conclude that  $H_{n,a}$ has no $(a,b)$-parity factors.
\end{proof}

\begin{lem}(Zhao, Huang and Wang\cite{ZHW})\label{lem::5}
	Let $n_1 \ge n_2 \ge \cdots \ge n_q$ be positive integers and $s$, $q$ be positive integers. Let $n=s+\sum _{i=1}^{q}n_i$, we have
	\begin{equation}
		\rho (K_{s}\nabla (K_{n_{1}}\cup K_{n_{2}}\cup \cdots\cup K_{n_{q}}) )\leq \rho (K_{s}\nabla (K_{n-s-q+1}\cup(q-1)K_{1})),
		\nonumber
	\end{equation}
	where the equality holds if and only if $(n_{1},n_{2},\cdots,n_{q})=(n-s-q+1,1,\cdots,1)$.
\end{lem}

    \begin{lem}\label{lem::6}
    	Let $n$ and $s$ be positive integers with $n \ge 4s+1$. Then 
    	$$
    	\rho (K_s \nabla (K_{n-s-3} \cup 3K_1 )) < n-2.
    	$$
    \end{lem}
    \begin{proof}
    	Suppose $L_{n,s}=K_s \nabla (K_{n-s-3} \cup 3K_1 )$. Let $V_{1}=V(K_s)$, $V_2=V(3K_1)$ and $V_3=V(K_{n-s-3})$. It is easy to see that the partition  $\Pi: V(L_{n,s})=V_1\cup V_2\cup V_3$ is an equitable partition of $L_{n,s}$, and the matrix $B_{\Pi}$ can be written as
    	$$
    	B_\Pi=\left(\begin{array}{ccccccc}
    		s-1&3&n-s-3\\
    		s&0&0\\
    		s&0&n-s-4\\
    	\end{array}\right).
    	$$
    	Let $f(x)$ denote the characteristic polynomial of $B_\Pi$. By a simple calculation, we have
    	$$
    	\begin{aligned}
    		f(n-2)&=|(n-2)I-B_\Pi|\\
    		&= 2n^{2}-6n-3s^{2}-6s+4\\
            &\ge 2(4s+1)^{2}-6(4s+1)-3s^{2}-6s+4\\
    		&\ge 29s^{2}-14s\\
    		&> 0.
    	\end{aligned} 
    	$$
    	Then we claim that $\lambda_1(B_\Pi)< n-2$. If not, since $f(n-4)=-3s^2<0$, we have $\lambda_3(B_\Pi)>n-4$, and hence $\lambda_1(B_\Pi)+\lambda_2(B_\Pi)+\lambda_3(B_\Pi)> 3n-12$. On the other hand,  $\lambda_1(B_\Pi)+\lambda_2(B_\Pi)+\lambda_3(B_\Pi)=\mathrm{trace}(B_\Pi)=n-4$, we obtain a contradiction. Therefore, by Lemma \ref{lem::1},
    	$$
    	\rho(L_{n,s})=\lambda_1(B_\Pi)< n-2.
    	$$
    \end{proof}

    By using a similar way as in the proof of Lemma \ref{lem::6}, we can obtain the following conclusion.
    \begin{lem}\label{lem::4}
        Let $n \ge 4$ be a positive integer. Then
        $$
        \rho (K_1 \nabla (K_{n-3} \cup 2K_1)) < n-2.
        $$
    \end{lem}

    Fiedler and Nikiforov \cite{FN} gave a spectral radius condition for graphs to have a Hamilton path or Hamilton cycle. 
    
    \begin{lem}(Fiedler and Nikiforov \cite{FN})
        Let $G$ be a graph of order $n$. If
        \begin{equation}\label{equ::2}
            \rho (G) \ge n-2,
        \end{equation}
        then $G$ contains a Hamiltonian path unless $G \cong H_{n,1}$. If the inequality (\ref{equ::2}) is strict, then $G$ contains a Hamiltonian cycle unless $G \cong H_{n,2}$.
    \end{lem}

    Note that the Hamilton path of even order contains a $1$-factor and the Hamilton cycle is a special $2$-factor. Thus we have the following result immediately.

    \begin{lem}\label{lem::7}
        Let $G$ be a graph of order $n$. 
    \begin{enumerate}[(i)]
        \item If $\rho (G) \ge n-2$ and $n$ be even, then $G$ has a $1$-factor unless $G \cong H_{n,1}$.
        \item If $\rho (G) > n-2$, then $G$ has a $2$-factor unless $G \cong H_{n,2}$.
	\end{enumerate}
    \end{lem}

\section{Proof of the main result}
	
In this section, we shall give the proof of Theorem \ref{thm::5}.

    {\flushleft \it Proof of Theorem \ref{thm::5}.} If $b=a=1$, according to assumption, we have $\rho (G) \ge \rho (H_{n,1}) = n-2$ and $n$ is even. By Lemma \ref{lem::7} and Lemma \ref{lem::3}, $G$ has a $1$-factor, i.e., $[1,1]$-parity factor unless $G \cong K_1\cup K_{n-1}=H_{n,1}$. If $b=2$, note that $a$  and $b$ have the same parity, so $a=2$. Thus, we can also obtain the conclusion by Lemma \ref{lem::7} and Lemma \ref{lem::3}. Suppose that $b \ge 3$. Since $a\geq 1 $, by Lemma \ref{lem::3}, we see that $H_{n,a}$ has no $(a,b)$-parity factors. For this reason, we always suppose that $G\ncong H_{n,a}$. By assumption, we have $\rho(G) \geq \rho(H_{n,a}) \ge \rho(K_{n-1})=n-2$ because $K_{n-1}$ is a proper subgraph of $H_{n,a}$. Now we claim that  $G$ is a connected graph. If $G$ is not a connected graph, we assume that ${G_1,G_2,\dots,G_k}$ ($k \ge 2$) denote the connected components of $G$. As each component of $G$ would be a subgraph of $K_{n-1}$, we have $\rho (G)=\max \{\rho (G_1),\rho (G_2),\dots,\rho (G_k)\} \leq \rho(K_{n-1}) = n-2$. This implies that $\rho (G)=n-2$ , hence $G\cong K_1\cup K_{n-1}=H_{n,1}$. If $a=1$, then we immediately obtain a contradiction because  $G\ncong H_{n,a}$. If $a>1$, then $H_{n,a}$ is connected, and $\rho(G) \geq \rho(H_{n,a}) >\rho(K_{n-1})=n-2$, which is impossible.
    
    Suppose to the contrary that $G$ has no  $(a,b)$-parity factors. By Corollary \ref{cor::1}, there exists two disjoint subsets $S$, $T$ of $V(G)$ such that
	\begin{equation}\label{equ::5}
		\eta (S,T) = b|S|-a|T|+\sum_{x \in T}d_{G-S}(x) -q(S,T) \leq -1,
	\end{equation}
    where $q(S, T)$ is defined in \eqref{equ::1}. Let $s=|S|$, $t=|T|$ and $q=q(S,T)$. We have the following claim.
    
    \begin{claim}\label{claim::1}
    	$e(\overline{G}) \leq n-2$.
    \end{claim}

    {\flushleft \it Proof of Claim 
	\ref{claim::1}.} 
    	Combining $\rho(G) \ge n-2$ with Lemma \ref{lem::2}, we have $e(G)\geq \binom{n-1}{2} +1$, and so $e(\overline{G}) \leq n-2$. This proves Claim \ref{claim::1}. 	 \qed\vspace{3mm}

  We consider the following three situations.
    
    {\flushleft {\it Case 1.} $t=0$.}
    
    In this situation, we suppose consider the following two cases.
    
    {\flushleft {\it Subcase 1.1.} $s=0$.}
    
    According to (\ref{equ::5}), we have $\eta(S,T)=-q \leq -1$, then $q \ge 1$. Note that $G$ is connected, so we have $q=1$. According to the definition of $q(S,T)$, we see that $na$ is odd, which is impossible.
    
    {\flushleft {\it Subcase 1.2.} $s \ge 1$.}
    
     In this situation, if $q \le 3$, then we have $bs \le q-1 \le 2$ by (\ref{equ::5}), contrarying to $bs \ge 3$. If $q \ge 4$, let $l_1$, $l_2$, ..., $l_q$ be the number of vertices of these q components of $G-S$ respectively, and let $l_{q+1}$, $l_{q+2}$, $\cdots$, $l_{q'}$ be the number of vertices of remaining components of $G-S$ respectively. We consider that graph $H$ satisfys $H \cong K_s \nabla (K_{n_1} \cup K_{n_2} \cup \dots \cup K_{n_q})$, where $n_1=l_1+l_{q+1}+l_{q+2}+ \cdots +l_{q'}$, $n_2=l_2$, $\cdots$, $n_q=l_q$. Thus $G$ is a spanning subgraph of $H$. Note that $3s \le bs \le q-1 \le n-s-1$, hence $n \ge 4s+1$. By Lemma \ref{lem::5} and Lemma \ref{lem::6}, we obtain 
    $$
    \begin{aligned}
    \rho(G) &\le \rho(H)\\
    &\leq \rho(K_s \nabla (K_{n_{1}}\cup K_{n_{2}} \cup \cdots \cup K_{n_{q}})) \\
    &\le \rho (K_{s}\nabla (K_{n-s-q+1}\cup(q-1)K_{1}))\\ 
    &\leq \rho(K_{s}\nabla (K_{n-s-3}\cup3K_{1}))\\
    &< n-2,
    \end{aligned}
    $$ 
    a contradiction.

  {\flushleft {\it Case 2.} $t=1$.}
  
  In this situation, we suppose $T = \{ x_0 \}$, and consider the following two cases.
  
  {\flushleft {\it Subcase 2.1.} $s=0$.}
  
  According to (\ref{equ::5}), we have $\eta(S,T)=-a+d_{G}(x_0)-q \leq -1$, then $d_{G}(x_0) \leq a+q-1$. 
  
  If $q \geq 3$, we consider that graph $H \cong K_1 \nabla (K_{n_1} \cup K_{n_2} \cup \dots \cup K_{n_q})$. Then $G$ is a spanning subgraph of $H$. Hence, by Lemma \ref{lem::5} and Lemma \ref{lem::4}, we also have $\rho(G) \le \rho(H) \leq  \rho(K_1 \nabla (K_{n-3} \cup 2K_1)) < n-2$, a contradiction. 
  
  If $1 \le q \le 2$, according to (\ref{equ::5}), we have $d_{G}(x_0) \leq a+1$. Let $T'=V(G)\setminus(S\cup T)$. If the number of components in $G[T']$ is greater than or equal to $2$, we have
  \begin{equation}
  	\begin{aligned}
  		e(G)&=e(S)+e(S,T)+e(S,T')+e(T)+e(T,T')+e(T')\\
  		&=e(T')+e(T,T')\\
  		&\le \frac{(n-2)(n-3)}{2}+a+1.\\
  	\end{aligned}
  	\nonumber
  \end{equation}
  Since $n>a+\frac{5}{2}$, by Lemma \ref{lem::2}, we obtain
  \begin{equation}
  	\begin{aligned}
  		\rho(G)&\le \sqrt{2e(G)-n+1}\\
  		&\le \sqrt{(n-2)(n-3)+2a-n+3}\\
  		&= \sqrt{n^{2}-6n+2a+9}\\
  		&= \sqrt{(n-2)^{2}-2(n-(a+\frac{5}{2}))}\\
  		&< n-2\\
  		&\leq \rho(H_{n,a}),
  	\end{aligned}
  	\nonumber
  \end{equation}
  contrary to our assumption.	Suppose that the number of components in $G[T']$ is equal to $1$, hence $q=1$. In this situation, we claim that $d_{G}(x_0) \leq a-1$. If $d_{G}(x_0)=a$, then we have $a|V(C)|+e_{G}(V(C),T) = an\equiv 1 \pmod 2$, a contradiction. Hence $G$ is a spanning subgraph of $H_{n, a}$, and we induce that $\rho(G)< \rho(H_{n,a})$ because of $G\ncong H_{n,a}$, which is a contradiction. 
  
  If $q=0$, then $d_{G}(x_0) \leq a-1$. Hence $G$ is a spanning subgraph of $H_{n, a}$, $\rho(G)< \rho(H_{n,a})$, which is a contradiction.

{\flushleft {\it Subcase 2.2.} $s \ge 1$.}
 	
  	According to (\ref{equ::5}), we have $d_{G}(x_0) \leq a-1+q-(b-1)s$. 
  	
  	If $q \leq (b-1)s$, we have  $d_{G}(x_0) \leq a-1$. By analyzing similarly to the above, we obtain a contradiction.
  	
  	If $q \geq (b-1)s+1$, by Claim \ref{claim::1}, there are at most $n-2-\frac{1}{2} \sum_{i=1}^{q}|V(C_i)|(n-s-1-|V(C_i)|)$ edges not in $E_G(T,V(G) \backslash (S \cup T))$. Hence,
  	\begin{equation}
  		\begin{aligned}
  			d_{G-S}(x_0)&\geq n-s-1-(n-2-\frac{1}{2} \sum_{i=1}^{q}|V(C_i)|(n-s-1-|V(C_i)|))\\
  			&=1-s+\frac{1}{2} \sum_{i=1}^{q}|V(C_i)|(n-s-1-|V(C_i)|)\\
  			&\geq 1-s+\frac{1}{2}\sum_{i=1}^{q}(n-s-1-|V(C_i)|)\\
  			&=1-s+\frac{1}{2}(q-1)(n-s-1)\\
  			&=\frac{1}{2}(q(n-s-1)-n-s+3).
  		\end{aligned}
        \nonumber
  	\end{equation}
  	Note that $n \ge q+s+1 \ge bs+2 $ and $q\geq (b-1)s+1$. By (\ref{equ::5}), we have
  	\begin{equation}
  		\begin{aligned}
  			\eta(S,T)&=bs-a+d_{G-S}(x_0)-q\\
  			&\geq bs-a+\frac{1}{2}(q(n-s-1)-n-s+3)-q\\
  			&=bs-a+\frac{1}{2}(q(n-s-3)-n-s+3)\\
  			&\ge \frac{1}{2}((b-1)(ns-s^2-s)-2a)\\
  			&\geq \frac{1}{2}((b-1)((b-1)s^2+s)-2a)\\
  			&\ge \frac{1}{2}(b(b-1)-2a)\\
  			&\ge 0,
  		\end{aligned}
  		\nonumber
  	\end{equation}
  	contrary to our assumption.

{\flushleft {\it Case 3.} $t\geq 2$.}

In this situation, we claim that $t \le 2a+2$. By contradiction, suppose that $t \ge 2a+3$. According to (\ref{equ::5}), we have $\eta(S,T)=bs-at+\sum_{x\in T}d_{G-S}(x)-q(S,T)\le-1$. Let $T'=V(G)\setminus(T \cup S)$, we obtain
\begin{equation}
	\begin{aligned}
		e(G)&=e(S)+e(S,T)+e(S,T')+e(T)+e(T,T')+e(T')\\
		&\le \frac{s(s-1)}{2}+st+s(n-s-t)+\sum_{x \in T}d_{G-S}(x)+\frac{(n-s-t)(n-s-t-1)}{2}\\
		&\le \frac{s(s-1)}{2}+st+s(n-s-t)+(at+q-bs-1)+\frac{(n-s-t)(n-s-t-1)}{2}\\
		&= \frac{(n-2)^{2}-n(2t-3)+t^{2}+(2s+1+2a)t+2q-2bs-6}{2}.
		\nonumber
	\end{aligned}
\end{equation}
Since $n \ge s+t$ and $t \ge 2a+3$, by Lemma \ref{lem::2}, we have
\begin{equation}
	\begin{aligned}
		\rho(G)&\le \sqrt{2e(G)-n+1}\\
		&\le  \sqrt{(n-2)^{2}-n(2t-3)+t^{2}+(2s+1+2a)t+2q-2bs-6-n+1}\\
		&\le  \sqrt{(n-2)^{2}-n(2t-3)+t^{2}+(2s+1+2a)t+2(n-s-t)-2bs-n-5}\\
		&=\sqrt{(n-2)^{2}-n(2t-4)+t^{2}+(2s-1+2a)t-(2b+2)s-5}\\
		&\le \sqrt{(n-2)^{2}-(2t-4)(s+t)+t^{2}+(2s-1+2a)t-(2b+2)s-5}\\
		&= \sqrt{(n-2)^{2}-t^{2}+(3+2a)t-(2b-2)s-5}\\
		&\le \sqrt{(n-2)^{2}-(t^{2}-(3+2a)t+5)}\\
		&< n-2,
		\nonumber
	\end{aligned}
\end{equation}
contrary to our assumption.     \qed\vspace{3mm}

Next we consider $q$ with two subcases.
{\flushleft {\it Subcase 3.1.} $q \ge 2$.}

By Claim \ref{claim::1}, there are at most $n-2- \frac{1}{2}\sum_{i=1}^q|V(C_i)|(n-s-t-|V(C_i)|)$ edges not in $E_G(T,V(G) \backslash (S \cup T)) \cup E_G(T)$, and therefore,
\begin{equation}
	\begin{aligned}
		\sum_{x\in T}d_{G-S}(x)&\geq (n-1-s)t-2(n-2- \frac{1}{2}\sum_{i=1}^q|V(C_i)|(n-s-t-|V(C_i)|))\\
		&\ge (n-1-s)t-2n+4+q(n-t-s)-\sum_{i=1}^q|V(C_i)|)\\
		&=(n-1-s)t-2n+4+(q-1)(n-t-s).
	\end{aligned}
    \nonumber
\end{equation}
Hence, by (\ref{equ::5}), we have
\begin{equation}
	\begin{aligned}
		\eta(S,T)&= bs-at+\sum_{x \in T}d_{G-S}(x)-q\\
		&\geq bs-at+(n-1-s)t-2n+4+(q-1)(n-t-s)-q\\
		&= bs-at+(n-1-s)t-2n+3+(q-1)(n-t-s-1)\\
		&\geq bs-at+2(n-1-s)-2n+3+n-s-t-1\\
		&= n+(b-3)s-(a+1)t\\
		&\geq n-(a+1)t\\
		&\geq n-(a+1)(2a+2)\\
		&\geq 0,
		\nonumber
	\end{aligned}
\end{equation}
contrary to our assumption.

{\flushleft {\it Subcase 3.2.} $q \le 1$.}

In this situation, we can obtain $\sum_{x \in T}d_{G-S}(x) \le at-bs$. Hence, $\sum_{x \in T}d_{G}(x) \le at+ts-bs$. If $t<b$, then $\sum_{x \in T}d_{G}(x) < at$. There must be a point $x$ in the graph $G$ belonging to $T$ such that $d_{G}(x) \le a-1$. Thus we can conclude $G$ is a spanning subgraph of $H_{n,a}$. As $G \ncong H_{n,a}$, we have $\rho(G) < \rho(H_{n,a})$, contrary to our assumption. If $t \geq b$, by analyzing similarly to the above, we have $\sum_{x \in T}d_{G-S}(x) \geq (n-1-s)t-2(n-2)$. Hence, by (\ref{equ::5}), 
\begin{equation}
	\begin{aligned}
		\eta(S,T)&= bs-at+\sum_{x \in T}d_{G-S}(x)-q\\
            &\geq bs-at+(n-1-s)t-2n+4-1\\
		&\geq bs-at+(n-1-s)b-2n+3\\
		&= (b-2)n-at-b+3\\
		&\ge n-at-b+3 \\
		&\ge n-a(2a+2)-b+3\\
		&\geq 0,
		\nonumber
	\end{aligned}
\end{equation}
contrary to our assumption.

Therefore, we conclude that $G$ has an $(a,b)$-parity factor, and the result follows. \qed


\begin{thebibliography}{99} 
	
\bibitem{AK}  J. Akiyama, M. Kano, Factors and Factorizations of Graphs: Proof Techniques in Factor Theory, vol. 2031, Springer, 2011.
	

\bibitem{A} R. P. Anstee, Simplified existence theorems for $(g, f)$-factors, Discrete Appl. Math. 27 (1990) 29--38.


\bibitem{BH} A.E. Brouwer, W.H. H$\mathtt{a}$emers, Spectra of Graphs, Springer, Berlin, 2011.

\bibitem{CHO} E. Cho, J. Hyun, S. O, J. Park, Sharp conditions for the existence of an even $[a,b]$-factor in a graph, Bull. Korean Math. Soc. 58 (1) (2021) 31--46.

\bibitem{Co} G. Cornu\'ejols, General factors of graphs, J. Combin. Theory Ser. B 45 (1988) 185--198.





\bibitem{FLL} D. Fan, H. Lin, H. Lu, Spectral radius and $[a, b]$-factors in graphs, Discrete Math. 345 (2022) 112892.

\bibitem{FN} M. Fiedler, V. Nikiforov, Spectral radius and Hamiltonicity of graphs, Linear Algebra Appl. 432 (9) (2010) 2170--2173.



\bibitem{GR} C. Godsil, G. Royle, Algebraic Graph Theory, Graduate Texts in Mathematics, vol. 207, Springer-Verlag, New York, 2001.

\bibitem{Ho} Y. Hong, A bound on the spectral radius of graphs, Linear Algebra Appl. 108 (1988) 135--139.


\bibitem{KT} M.  Kano, N.  Tokushige, Binding numbers and $f$-factors of graphs, J. Combin. Theory Ser. B 54 (1992) 213--221.

\bibitem{KO}  S. Kim, S. O, J. Park, H. Ree, An odd $[1,b]$-factor in regular graphs from eigenvalues, Discrete Math. 343 (8) (2020) 111906.




\bibitem{LL} H. Liu, H. Lu, A degree condition for a graph to have $(a,b)$-parity factors, Discrete Math. 341 (2018) 244--252.

\bibitem{LZ} G. Liu, L. Zhang, Fractional $(g,f)$-factors of graphs, Acta Math. Sci. 21 (4) (2001) 541--545.

\bibitem{LWY} H. Lu, Z. Wu, X. Yang, Eigenvalues and $[1, n]$-odd factors, Linear Algebra Appl. 433 (2010) 750--757.

\bibitem{Lo} L. Lov\'{a}sz, The factorization of graphs. \Rmnum{2}, Acta Math. Sci. Hungar. 23 (1972) 223--246. 

\bibitem{Lo1} L. Lov\'{a}sz, Subgraphs with prescribed valencies, J. Comb. Theory 8 (1970) 391--416.









\bibitem{OS2} S. O, Eigenvalues and $[a,b]$-factors in regular graphs, J. Graph Theory 100 (3) (2022) 458--469.

\bibitem{St}D. Stevanovi\'{c}, Spectral Radius of Graphs, Academic Press, London, 2015.


\bibitem{T} W. T. Tutte, The factors of graphs, Canad. J. Math. 4 (1952) 314--328.

\bibitem{TT} W. T. Tutte, Graph factors, Combinatorica 1 (1981) 79--97.

\bibitem{WXH} B. Wu, E. Xiao, Y. Hong, The spectral radius of trees on $k$ pendant vertices, Linear Algebra Appl. 395 (2005) 343–349.

\bibitem{WZ} J. Wei, S. Zhang, Proof of a conjecture on the spectral radius condition for $[a, b]$-factors, Discrete Math. 346 (2023) 113269. 


\bibitem{Xi} L. Xiong, Characterization of forbidden subgraphs for the existence of even factors in a graph, Discrete Appl. Math. 223 (2017) 135--139.

\bibitem{ZHW} Y. Zhao, X. Huang, Z. Wang, The $A_{\alpha}$-spectral radius and perfect matchings of graphs, Linear Algebra Appl. 631 (2021) 143--155.



\end{thebibliography}
\end{document}